\documentclass[12pt]{amsart}

\usepackage{amsmath,amssymb,latexsym,amsthm,newlfont,enumerate}

\usepackage{graphicx}
\usepackage[all]{xy}

\theoremstyle{plain}

\newtheorem{thm}{Theorem}[section]

\newtheorem{pro}[thm]{Proposition}

\newtheorem{lem}[thm]{Lemma}

\newtheorem{cla}[thm]{Claim}
\newtheorem{cor}[thm]{Corollary}
\newtheorem{con}[thm]{Conjecture}

\theoremstyle{definition}

\newtheorem{dfn}[thm]{Definition}

\newtheorem{rem}[thm]{Remark}

\newtheorem{exa}[thm]{Example}

\theoremstyle{remark}

\newcommand{\Z}{\mathbb{Z}}
\newcommand{\N}{\mathbb{N}}

\newcommand{\R}{\mathbb{R}}
\newcommand{\Q}{\mathbb{Q}}
\newcommand{\PS}{\mathbb{P}}

\newcommand{\OO}{\mathcal{O}}

\newcommand{\id}{\mathrm{id}}
\newcommand{\vphi}{\varphi}

\newcommand{\dashto}{\dashrightarrow}
\newcommand{\lto}{\longrightarrow}
\newcommand{\mcal}{\mathcal}

\DeclareMathOperator{\inte}{int}

\DeclareMathOperator{\codim}{codim}

\DeclareMathOperator{\Exc}{Exc}
\DeclareMathOperator{\mult}{mult}

\DeclareMathOperator{\GL}{GL}
\DeclareMathOperator{\Aut}{Aut}
\DeclareMathOperator{\Bir}{Bir}
\DeclareMathOperator{\Supp}{Supp}

\DeclareMathOperator{\Fix}{Fix}

\DeclareMathOperator{\Proj}{Proj}

\DeclareMathOperator{\cont}{cont}

\DeclareMathOperator{\PsAut}{PsAut}

\DeclareMathOperator{\Mob}{Mob}
\DeclareMathOperator{\Pic}{Pic} 

\DeclareMathOperator{\Div}{Div}

\DeclareMathOperator{\NEb}{\overline{\mathrm{NE}}}

\DeclareMathOperator{\Nef}{Nef}

\DeclareMathOperator{\Effb}{\overline{\mathrm{Eff}}}
\DeclareMathOperator{\Eff}{\mathrm{Eff}}
\DeclareMathOperator{\Mov}{Mov}
\DeclareMathOperator{\B}{Big}

\DeclareMathOperator{\ddiv}{div}

\begin{document}
\title{Around and beyond the canonical class}

\author{Vladimir Lazi\'c}
\address{Mathematisches Institut, Universit\"at Bayreuth, 95440 Bayreuth, Germany}
\email{vladimir.lazic@uni-bayreuth.de}

\thanks{Many thanks to P.~Cascini, A.~Corti, K.~Frantzen, D.~Greb, A.-S.~Kaloghiros, J.~Koll\'ar, A.~K\"uronya, Th.~Peternell and S.~Weigl for many useful comments and discussions. I was supported by the DFG-Forschergruppe 790 ``Classification of Algebraic Surfaces and Compact Complex Manifolds".}

\begin{abstract}
This survey is an invitation to recent developments in higher dimensional birational geometry.
\end{abstract}

\maketitle
\setcounter{tocdepth}{1}
\tableofcontents

\section{Introduction}

This survey is an invitation to recent techniques related to the Minimal Model Program. My goal is to persuade you that the MMP, at least in some of its parts, is not a subject to be afraid of any more, and that it can be swallowed by a hungry postgraduate student within one (advanced) course.

Indeed, in writing this paper, I had in particular such a student in mind. I deliberately tried not to be too pedantic, so that the material can be widely accessible, and that the exposition can be as clear as possible. 

Until recently, the proofs of foundational results in the MMP were of such technical complexity that they remained opaque to all but a handful of experts. This state of affair is changing due to the emergence of a new outlook on the subject. This new outlook is the topic of this paper.

The Minimal Model Program has seen tremendous progress in the last decade, which is measurable both in scope of the results achieved, as well as in the depth of our understanding of the subject. The seminal paper \cite{BCHM}, building on earlier results of Mori, Reid, Kawamata, Koll\'ar, Shokurov, Siu, Corti, Nakayama and many others, settled many results and advanced hugely our knowledge of the theory. The paper \cite{BCHM} builds upon, in some sense, classical theory, starting with the Cone theorem on our preferred variety, and employing a complicated induction to construct a sequence of surgery operations, which is then shown to terminate and yield a desired birational model which has exceptional properties. This, in turn, provided a proof of one of the most influential conjectures in Algebraic Geometry in the last 50 years, the finite generation of the canonical ring, posed in Zariski's famous paper \cite{Zar62}. A more general version of this result is Theorem \ref{thmA} below.

On the other hand, it has recently become clear that we can look at the picture the other way round. In \cite{Laz09,CaL10}, Theorem \ref{thmA} was proved directly and without the MMP, only by using induction on the dimension and the Kawamata-Viehweg  vanishing. The proof of this result is not the topic here, as it was clearly surveyed in \cite{Corti11,CaL12}. In this paper, I take Theorem \ref{thmA} as a \emph{black box}, and build upon it. 

The moral of the story is that this result, together with the right tools which are developed in Section \ref{sec:graded}, implies (almost) everything we know about the MMP in a clearer and quicker way. This was worked out in \cite{CoL10}, and is the content of Section \ref{sec:picture1} below. This section forms the basis for the discussion in the remainder of the paper, and it is important both from the motivational viewpoint, as well as in the scope of the techniques used.

Moreover, we will see below in Section \ref{sec:mmpbeyond} that the new outlook gives the right perspective to think about some other problems in the field. One of them is a (possibly more philosophical) question: what makes the canonical sheaf $\omega_X$ special, say on a smooth projective variety $X$. Ever since Riemann's work on curves in the 19th century, the importance of $\omega_X$ has been realised: in part because of the Riemann-Roch theorem, and in part because often it is very difficult to find reasonable and useful divisors on $X$. Of course, in the 20th century it was understood further that this line bundle is important because of Serre duality, Kodaira vanishing and so on. Therefore, it is logical to concentrate on $\omega_X$ as the centre point of classification, i.e.\ the MMP.

The class of varieties where the classical MMP works is huge -- in particular, all smooth varieties are covered. However, there are many singular varieties where the results cannot apply. Indeed, Example \ref{e_canonical} gives a normal projective variety for which no reasonable definition of the MMP attached to $\omega_X$ works. On the other hand, there are varieties, called Mori Dream Spaces, which possess a rich birational geometry similar to the classical MMP, but they need not necessarily fall into the class of singularities allowed by the classical MMP. I survey this type of varieties in Section \ref{sec:picture2}, drawing parallels to Section \ref{sec:picture1}, and this motivates what happens in Section \ref{sec:mmpbeyond}.

This begs the question whether we can formulate a framework which contains both the classical MMP and Mori Dream Spaces, and which constitutes, in some sense, the maximal class where a ``reasonable" birational geometry can be performed. Indeed, this was done in \cite{KKL12} by extending the techniques from \cite{CoL10}, and as I try to convince you in Section \ref{sec:mmpbeyond}, the result is surprisingly simple and appealing.

Finally, I close the paper with a discussion of a particular conjecture which aims to describe various cones in the space of divisors on Calabi-Yau manifolds. The Cone conjecture, due to Morrison and Kawamata, is a still pretty mysterious prediction, but I argue that it is consistent with probably the most important outstanding conjecture in birational geometry, the Abundance conjecture. On the way, we will see how the material from Section \ref{sec:graded} applies nicely to show that parts of these cones have a particularly good shape.

Throughout the paper, all varieties are normal and projective, and everything happens over the complex numbers. I follow notation and conventions from \cite{Laz04}, and anything which is not explicitly defined here, can be found there.
\section{Graded rings of higher rank}\label{sec:graded}

In this section I make a brief introduction to divisorial rings, with  particular accent on the higher rank case. It has only recently become clear that, even though at first they seem more complicated than rings graded by $\N$, once you are ready to make a brave step and develop (or are just simply willing to accept) the necessary theory, then most proofs become much easier and more conceptual.

To start with, let $X$ be a $\Q$-factorial projective variety, and let $D$ be a $\Q$-divisor on $X$. Then we define the global sections of $D$ by 
\[
H^0(X,D)=\{f\in k(X)\mid \ddiv f+D \geq 0 \}.
\]
Note that, even though $D$ might not be an integral divisor, this makes perfect sense, and that $H^0(X,D)=H^0(X,\lfloor D\rfloor)$, where the latter $H^0$ is the vector space of global sections of the standard divisorial sheaf $\OO_X(\lfloor D\rfloor)$.  This is compatible with taking sums: in other words, there is a well-defined multiplication map 
$$H^0(X,D_1)\otimes H^0(X,D_2)\to H^0(X,D_1+D_2).$$
  
Therefore, if we are given a bunch of $\Q$-divisors $D_1,\dots,D_r$ on $X$, we can define the corresponding \emph{divisorial ring} as
\[
\mathfrak R=R(X;D_1, \dots, D_r)=\bigoplus_{(n_1,\dots, n_r)\in \N^r} H^0(X,
n_1D_1+\dots + n_rD_r).
\]
When $r=1$, then we usually say that the ring $R(X,D_1)$ is the \emph{section ring} of $D_1$.

Throughout this paper, there is a recurring assumption that rings that we study are finitely generated, and we will see that this assumption alone has far-reaching consequences. 

So say that we have a divisorial ring $\mathfrak R$ as above, and assume that it is finitely generated. Then we have a corresponding cone $\mcal C=\sum\R_+ D_i$ which sits in the space of $\R$-divisors $\Div_\R(X)$. Inside $\mcal C$, there is another, much more important cone -- the \emph{support} of $\mathfrak R$. This cone, $\Supp\mathfrak{R}$, is defined as the convex hull of all integral divisors $D\in\mcal C$ which have sections, i.e.\ $H^0(X,D)\neq0$. It is easily seen that $\Supp\mathfrak R$ is a rational polyhedral cone: indeed, pick generators $f_i$ of $\mathfrak R$, and let $E_i\in\mcal C$ be the divisors such that $f_i\in H^0(X,E_i)$. Then clearly $\Supp\mathfrak R=\sum\R_+E_i$.

\begin{exa}\label{e_cutkosky}
The first example when $\mathfrak R$ is finitely generated is when all $D_i$ are semiample divisors: indeed, this is an old result of Zariski \cite{Zar62}.

On the other hand, even on curves there are divisorial rings which are not finitely generated. Indeed, let $E$ be an elliptic curve, let $D$ be a non-torsion divisor of degree $0$, and let $A$ be an ample divisor on $E$. Then the $\N^2$-graded ring 
$$R(E;D,A)=\bigoplus_{(i,j)\in\N^2}R_{i,j}$$
is not finitely generated: it is easy to see that the support of this ring is equal to the set $(\R_+D+\R_+ A)\setminus \R_{>0}D$, and hence it is not a rational polyhedral cone.

This immediately yields a surface $Y$ and a line bundle $M$ on $Y$ whose section ring is not finitely generated: set $Y=\PS(\OO_E(D)\oplus\OO_E(A))$ and $M=\OO_Y(1)$. Then $H^0(Y,M^{\otimes k})\simeq\bigoplus_{i+j=k}R_{i,j}$, hence the section ring $R(Y,M)$ is not finitely generated by the argument above.
\end{exa}

The following lemma summarises the main tools when operating with finite generation of divisorial rings. The proof can be found in \cite[\S 2.4]{CaL10}.

\begin{lem}\label{lem:3}
Let $X$ be a $\Q$-factorial projective variety, and let $D_1,\dots,D_r$ be $\Q$-divisors such that the ring $R(X;D_1,\dots,D_r)$ is finitely generated.
\begin{enumerate}
\item If $p_1,\dots,p_r\in\Q_+$, then the ring $R(X;p_1D_1,\dots,p_rD_r)$ is finitely generated. 
\item Let $G_1,\dots,G_\ell$ be $\Q$-divisors such that $G_i\in\sum\R_+D_i$ for all $i$. Then the ring $R(X;G_1,\dots,G_\ell)$ is finitely generated.
\end{enumerate} 
\end{lem}

\subsection*{An important example}
It has become clear in the last several decades that sometimes varieties are not the right objects to look at -- often, it is much more convenient to look at pairs $(X,\Delta)$, where $X$ is a normal projective variety and $\Delta$ is a Weil $\Q$-divisor on $X$ such that $K_X+\Delta$ is $\Q$-Cartier. There are plenty of reasons for looking at these objects: they obviously generalise the concept of a ($\Q$-Gorenstein) variety (by taking $\Delta=0$), they are suitable for induction because of adjunction formula, they are closely related to open varieties $X\setminus\Supp\Delta$, and so on. It is difficult to name all the advantages of working in this setting, especially since the idea of pairs and their singularities had brewed for a very long time; a good place to find a thorough explanation of all this is \cite{KM98}.

Not all pairs are useful for us. We concentrate on a special kind of pairs, those that have \emph{klt singularities}. This means the following. First note that if $f\colon Y\to X$ is a log resolution of the pair $(X,\Delta)$, that is, $Y$ is a smooth variety and the support of the set $f^{-1}_*\Delta\cup\Exc f$ is a simple normal crossings divisor, then there exists a $\Q$-divisor $R$ on $Y$ such that
$$K_Y=f^*(K_X+\Delta)+R.$$
Then we say that $(X,\Delta)$ is klt, or that it has klt singularities, if all the coefficients of $R$ are bigger than ${-}1$. It can be shown that this does not depend on the choice of the resolution $f$.

This looks like a very mysterious condition. However, a good way to think about it is to assume from the start that $X$ is smooth, that $\Supp\Delta$ has simple normal crossings, and that all coefficients of $\Delta$ lie in the open interval $(0,1)$. It is a fun exercise to prove that such a pair indeed has klt singularities. In particular, smooth varieties $X$, viewed as pairs $(X,0)$, have klt singularities.

Also of importance for us is that this is an open condition, in the following sense. Say you have at hand a klt pair $(X,\Delta)$ with $X$ being $\Q$-factorial, and that you have an effective $\Q$-divisor $D$ on $X$. Then for all rational $0\leq\varepsilon\ll1$, the pair $(X,\Delta+\varepsilon D)$ is again klt. This is easy to see from the definition. 

Therefore, divisors of the form $K_X+\Delta$ are of special importance for us,  and they are called \emph{adjoint divisors}. A special case of the divisorial ring above is when all $D_i$ are (multiples of) adjoint divisors -- we then say that the ring $\mathfrak R$ is an {\em adjoint ring\/}. 

Now we are ready to state the most important example of a finitely generated divisorial ring.

\begin{thm}\label{thmA}
Let $X$ be a $\Q$-factorial projective variety, and let $\Delta_1,\dots,\Delta_r$ be big $\Q$-divisors such that all pairs $(X,\Delta_i)$ are klt.

Then the adjoint ring
\[
R(X;K_X+\Delta_1,\dots,K_X+\Delta_r)
\]
is finitely generated.
\end{thm}

This was first proved in \cite{BCHM} by employing the full machinery of the classical MMP: the idea is to prove that a certain version of the Minimal Model Program works, and then to deduce the finite generation as a consequence. However, as mentioned in the introduction, of importance here for us is that Theorem \ref{thmA} can be proved without the MMP, and this was done in \cite{Laz09,CaL10}. 

\subsection*{Asymptotic valuations}
We will see that finite generation of a divisorial ring $\mathfrak R$ has important consequences on the convex geometry of the cone $\Supp\mathfrak R$. We would like to relate the ring $\mathfrak R$ to the behaviour of linear systems $|D|$ for integral divisors $D\in\Supp\mathfrak R$. The way to achieve this is via asymptotic geometric valuations. 

Let $X$ be a $\Q$-factorial projective variety. Then each prime divisor $\Gamma$ on $X$ gives a valuation on the ring of rational function $k(X)$ as the order of vanishing at the generic point of $\Gamma$. This is not sufficient, as the behaviour of elements of $k(X)$ depends also on the higher codimension points. 

Therefore a {\em geometric valuation\/} $\Gamma$ on $X$ is any valuation on $k(X)$ which is given by the order of vanishing at the generic point of a prime divisor on some birational model $Y\to X$, and we denote the value of this valuation on a $\Q$-divisor $D$ by $\mult_\Gamma D$; in other words, we take into account exceptional divisors as well. 

Now, if $D$ is an \emph{effective} $\Q$-Cartier divisor, then the {\em asymptotic order of vanishing\/} of $D$ along $\Gamma$ is 
$$o_\Gamma (D)=\inf\{\mult_\Gamma D'\mid D\sim_\Q D'\geq0\};$$
put differently, if $\mult_\Gamma|kD|$ is the valuation at $\Gamma$ of a general element of the linear system $|kD|$, then 
$$o_\Gamma(D)=\inf\frac1k\mult_\Gamma|kD|$$
over all $k$ sufficiently divisible. 

It is straightforward to see that each $o_\Gamma$ is a homogeneous function of degree $1$, that 
$$o_\Gamma(D+D')\leq o_\Gamma(D)+o_\Gamma(D')$$ 
for every two effective $\Q$-divisors $D$ and $D'$, and that 
$$o_\Gamma(A)=0$$ 
for every semiample divisor $A$. The following is a basic result \cite{Nak04}:

\begin{lem}\label{lem:numericalBig}
Let $X$ be a $\Q$-factorial projective variety, and let $D$ and $D'$ be two big $\Q$-divisors on $X$ such that $D\equiv D'$. Then $o_\Gamma(D)=o_\Gamma(D')$. 
\end{lem}
\begin{proof}
I first claim that for any ample $\Q$-divisor $A$, we have 
$$o_\Gamma(D)=\lim_{\varepsilon\downarrow0}o_\Gamma(D+\varepsilon A).$$ 
To this end, note that by Kodaira's trick we can write $D\sim_\Q \delta A+E$ for some rational $\delta>0$ and an effective $\Q$-divisor $E$. Therefore
$$(1+\varepsilon)o_\Gamma(D)=o_\Gamma(D+\varepsilon\delta A+\varepsilon E)\leq o_\Gamma(D+\varepsilon\delta A)+\varepsilon o_\Gamma(E)\leq o_\Gamma(D)+\varepsilon o_\Gamma(E),$$
and we obtain the claim by letting $\varepsilon\downarrow0$.

Now, fix an ample divisor $A$ and a rational number $\varepsilon>0$. Since the divisor $D-D'+\varepsilon A$ is numerically equivalent to $\varepsilon A$, and thus ample, we have
$$o_\Gamma(D+\varepsilon A)=o_\Gamma\big(D'+(D-D'+\varepsilon A)\big)\leq o_\Gamma(D').$$
Letting $\varepsilon\downarrow0$ and applying the claim, we get $o_\Gamma(D)\leq o_\Gamma(D')$. The reverse inequality is analogous.
\end{proof}

Now we have all the theory needed to state the result which gives us the main relation between finite generation and the behaviour of linear systems.

\begin{thm}
\label{thm:ELMNP}
Let $X$ be a $\Q$-factorial projective variety, and let $D_1,\dots,D_r$ be $\Q$-divisors on $X$. Assume that the ring $\mathfrak R=R(X;D_1,\dots,D_r)$ is finitely generated. Then:
\begin{enumerate}
\item $\Supp \mathfrak R$ is a rational polyhedral cone,
\item if $\Supp \mathfrak R$ contains a big divisor, then all pseudo-effective divisors in $\sum\R_+ D_i$ are in fact effective,
\item there is a finite rational polyhedral subdivision $\Supp \mathfrak R=\bigcup \mcal{C}_i$ into cones of maximal dimension, such that $o_\Gamma$ is linear on $\mcal{C}_i$ for every geometric valuation $\Gamma$ over $X$,
\item there exists a positive integer $k$ such that $o_\Gamma(kD)=\mult_\Gamma|kD|$ for every integral divisor $D\in \Supp \mathfrak R$.
\end{enumerate}
\end{thm}

We already saw (1) before, and the proof of (2) is also very easy. I omit the proof of (3) and (4), but it is not too difficult once one sets up a good basis of algebra and convex geometry. This important result is contained in the proof of \cite[Theorem 4.1]{ELMNP}, and is merely extracted verbatim in \cite[Theorem 3.6]{CoL10}.

A simple, but as we will see important consequence is the following. 

\begin{lem}\label{lem:ords}
Let $X$ be a normal projective variety and let $D$ be an effective $\Q$-divisor on $X$. Then $D$ is semiample if and only if $R(X, D)$ is finitely generated and $o_{\Gamma}(D)=0$ for all geometric valuations $\Gamma$ over $X$.
\end{lem}

The proof is very simple: if $D$ is semiample, then we conclude by Example \ref{e_cutkosky}. Conversely, for every point $x\in X$, Theorem \ref{thm:ELMNP} implies that $x$ does not belong to the base locus of the linear system $|mD|$ for $m$ sufficiently divisible.

As a demonstration of the previous two results, we will see immediately how inside the cone $\Supp\mathfrak R$, all the cones that we can imagine behave nicely. Recall first that the \emph{movable cone} $\overline{\Mov}(X)$ is the closure of the cone spanned by movable divisors -- these are integral divisors whose base loci are of codimension at least $2$.

The following is effectively the proof of Mori's Cone theorem, see Section \ref{sec:picture1}. I borrow the proof from \cite{CoL10,KKL12}.

\begin{pro}\label{cor:5}
Let $X$ be a $\Q$-factorial projective variety and let $D_1, \dots, D_r$ be $\Q$-divisors on $X$. Assume that the ring $\mathfrak R=R(X;D_1, \dots, D_r)$ is finitely generated, and denote by $\pi\colon\Div_\R(X)\to N^1(X)_\R$ the natural projection. Then:
\begin{enumerate}
\item the cone $\Supp\mathfrak R\cap \pi^{-1}\bigl(\overline{\Mov} (X)\bigr)$ is rational polyhedral,
\item if $\Supp\mathfrak R$ contains an ample divisor, then the cone $\Supp\mathfrak R\cap \pi^{-1}\bigl(\Nef(X)\bigr)$ is rational polyhedral, and every element of this cone is semiample.
\end{enumerate} 
\end{pro}
\begin{figure}[htb]
\begin{center}
\includegraphics[width=0.4\textwidth]{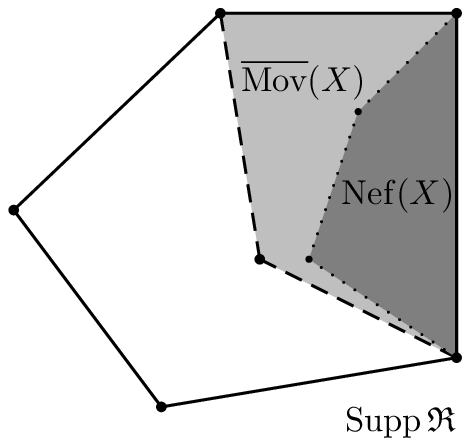}
\end{center}
\end{figure}
\begin{proof}
For every prime divisor $\Gamma$ on $X$, denote by $\mathcal C_\Gamma$ the span of the set of all $\Q$-divisors $D\in\Supp\mathfrak R$ such that $o_\Gamma(D)=0$. Note that $\Supp\mathfrak R\cap \pi^{-1}\bigl(\overline{\Mov} (X)\bigr)$ is the intersection of all $\mathcal C_\Gamma$ by Theorem \ref{thm:ELMNP}(4). 

Let $\Supp \mathfrak R=\bigcup\mcal{C}_i$ be a finite rational polyhedral subdivision as in Theorem~\ref{thm:ELMNP}. We may add all the faces of all the cones $\mcal C_i$ to the subdivision. To show (1), it is enough to prove that each $\mcal C_\Gamma$ is a union of some of $\mcal C_i$. But this follows once we notice that, if $\mcal C_\Gamma$ intersects the relative interior of some $\mcal C_\ell$, then $\mcal C_\ell\subseteq\mcal C_\Gamma$ since $o_\Gamma$ is a linear non-negative function on $\mcal C_\ell$. 

For (2), if the relative interior of $\mcal C_\ell$ contains an ample divisor, then $o_{\Gamma|\mcal{C}_\ell}\equiv0$ for every $\Gamma$ as above. Hence, every element of $\mcal{C}_\ell$ is semiample by Lemmas \ref{lem:3} and \ref{lem:ords}, and so $\mcal{C}_\ell\subseteq \Supp\mathfrak{R}\cap\pi^{-1} \bigl(\Nef(X)\bigr)$. Therefore, the cone $\Supp\mathfrak{R}\cap\pi^{-1} \bigl(\Nef(X)\bigr)$ is equal to the union of some $\mcal C_i$, which suffices.
\end{proof}

We also note the following crucial consequence of Theorem \ref{thm:ELMNP} \cite[Theorem 4.2]{KKL12}. It shows that the chamber decomposition as in Theorem \ref{thm:ELMNP} gives canonically birational contractions from our variety.

\begin{thm}\label{thm:decomposition}
Let $X$ be a $\Q$-factorial projective variety, and let $D_1,\dots,D_r$ be $\Q$-divisors on $X$. Assume that the ring $\mathfrak R=R(X;D_1,\dots,D_r)$ is finitely generated, and that $\Supp\mathfrak R$ contains a big divisor. Let $\Supp\mathfrak R= \bigcup \mcal C_i$ be a finite rational polyhedral decomposition as in Theorem \ref{thm:ELMNP}, and let $\mcal F_j$ be all the codimension $1$ faces of the cones $\mcal C_i$.
\begin{enumerate}
\item For each $i$, let $D_i$ be a Cartier divisor in the interior of $\mcal C_i$, and let $X_i=\Proj R(X,D_i)$. Then the variety $X_i$ and the birational map $\vphi_i\colon X\dashto X_i$ do not depend on the choice of $D_i$ (up to isomorphism). The map $\varphi_i$ is a contraction.
\item For each $j$, let $G_j$ be a Cartier divisor in the relative interior of $\mcal F_j$, and let $Y_j=\Proj R(X,G_j)$. If $\mcal F_j$ contains a big divisor, then the variety $Y_j$ and the birational map $\theta_j\colon X\dashto Y_j$ do not depend on the choice of $G_j$ (up to isomorphism). The map $\theta_j$ is a contraction.
\item If $\mcal F_j\subseteq\mathcal C_i$, then there is a birational morphism $\rho_{ij}\colon X_i\to Y_j$ such that the diagram 
\[
\xymatrix{ 
X \ar@{-->}[rr]^{\varphi_i} \ar@{-->}[dr]_{\theta_j} & \quad & X_i
  \ar[dl]^{\rho_{ij}}\\
\quad & Y_j & \quad
}\]
commutes.
\end{enumerate}
\end{thm}

\begin{proof}
I will only show (1) and (3), as the proof of (2) is analogous to that of (1). 

Theorem \ref{thm:ELMNP} implies that we can find a resolution $f\colon \widetilde{X} \to X$ and a positive integer $d$ such that $\Mob f^*(dD)$ is basepoint free for every Cartier divisor $D\in\Supp\mathfrak R$. Denote $M_i= \Mob f^*(dD_i)$. Then we have the induced birational morphism $\psi_i\colon \widetilde X \to X_i$, which is just the Iitaka fibration associated to $M_i$. A result of Reid \cite[Proposition 1.2]{Rei80} shows that the divisor $\Fix |f^*(dD_i)|$ (and also any $f$-exceptional divisor) is contracted by $\psi_i$ -- in other words, $\varphi_i$ is a contraction.

Let us show that the definition of $\varphi_i$ does not depend on the choice of $D_i$. Indeed, pick any other Cartier divisor $D_i'$ in the interior of $\mcal C_i$, and let $\psi_i'\colon\widetilde X\to \Proj R(X,D_i')$ be the corresponding map. There exists a Cartier divisor $D_i''$ in the interior of $\mcal{C}_i$, together with positive integers $r_i,r_i',r_i''$ such that 
$$r_iD_i= r_i' D_i'+r_i''D_i''.$$
Denoting $M_i'= \Mob f^*(dD_i')$ and $M_i''= \Mob f^*(dD_i'')$, then we have 
\begin{equation}\label{eq:mobiles}
r_iM_i= r_i' M_i'+r_i''M_i''
\end{equation} 
(since all functions $o_\Gamma$ are linear on $\mcal C_i$), and the divisors $M_i,M_i',M_i''$ are basepoint free. For any curve $C$ on $\widetilde X$ contracted by $\psi_i$ we have $M_i\cdot C=0$, hence equation \eqref{eq:mobiles} implies $M_i' \cdot C=0$, and so $C$ is contracted by $\psi_i'$. Reversing the roles of $D_i$ and $D_i'$, we obtain that $\psi_i$ and $\psi_i'$ contract the same curves, therefore they are the same map up to isomorphism.

The same method proves (3).
\end{proof}

I finish this section with a simple consequence of Lemma \ref{lem:numericalBig} and Theorem \ref{thm:ELMNP}, which will be crucial in Section \ref{sec:mmpbeyond}.

\begin{lem}\label{lem:equal proj}
Let $X$ be a $\Q$-factorial projective variety, and let $D_1$ and $D_2$ be big $\Q$-divisors such that $D_1\equiv D_2$. Assume that the rings $R(X,D_i)$ are finitely generated, and consider the maps $\vphi_i\colon X\dashto \Proj R(X,D_i)$.

Then there exists an isomorphism $\eta\colon \Proj R(X,D_1)\lto \Proj R(X,D_2)$ such that $\vphi_2=\eta\circ\vphi_1$.
\end{lem}
\begin{proof}
By passing to a resolution and by Theorem \ref{thm:ELMNP}, we may assume that there is a positive integer $k$ such that $\Mob(kD_i)$ are basepoint free, and that each $\varphi_i$ is given by the linear system $|\Mob(kD_i)|$. By Lemma \ref{lem:numericalBig} we have $\Mob(kD_1)\equiv \Mob(kD_2)$, hence $\vphi_1$ and $\vphi_2$ contract the same curves.
\end{proof}
\section{Picture 1: Classification}\label{sec:picture1}

I review briefly the ``classical" Minimal Model Program, concentrating on parts which are important in what follows. There are many well-written surveys and books on the topic, and if needed, you can consult \cite{KM98} and references therein. It is important to point out that many concepts which are related or grew out of the MMP, and that we consider \emph{natural} or \emph{given} because they fit beautifully into many corners of algebraic geometry (such as nefness, the canonical ring, minimal and canonical models and so on), took a long time to conceive. In order to fully appreciate this formative process, I urge you to read the wonderful semi-autobiographical survey \cite{Rei00}, and also \cite[\S 9]{Mor87}. 

For the sake of simplicity and clarity, I state all the results for smooth varieties, but note that with minor changes they hold for pairs with klt singularities. 

So say you have at hand a smooth $n$-dimensional projective variety $X$. As mentioned in the introduction, it is reasonable to concentrate on the canonical divisor $K_X$ as the central object of our study. On the other hand, having ample divisors on a projective variety $X$ is extremely important: they give embeddings of $X$ into some projective space, and their cohomological and numerical properties are as nice as you can hope for.

Assume that $K_X$ is pseudo-effective. Then, a reasonable question to pose is: 
\begin{center}
\emph{Is there a birational map $f\colon X\dashrightarrow Y$ such that the divisor $f_*K_X$ is ample?}
\end{center}
Here the map $f$ should not be just any birational map, but a birational contraction -- in other words, $f^{-1}$ should not contract divisors. This is an important condition since the variety $Y$ should be in almost every way simpler than $X$; in particular, some of its main invariants, such as the Picard number, should not increase. Likewise, we would like to have $K_Y=f_*K_X$, and this will almost never happen if $f$ extracts divisors (take, for instance, an inverse of almost any blowup).

Further, we impose that $f$ should preserve sections of all positive multiples of $K_X$. This is also important, since global sections are something we definitely want to keep track of, if we want the divisor $K_Y=f_*K_X$ to bear any connection with $K_X$. Another way to state this is as follows. Consider the \emph{canonical ring} of $X$:
$$R(X,K_X)=\bigoplus_{m\in\N}H^0(X,mK_X).$$
Then we require that $f$ induces an isomorphism between $R(X,K_X)$ and $R(Y,K_Y)$. Apart from the relation to Zariski's conjecture that was mentioned in the introduction, this is also fundamental in the construction of the moduli space of canonically polarised varieties; for an introduction to this beautiful topic, upon which I do not touch any more in these notes, see \cite[Part III]{HK10}.

We immediately see that the answer to the question above is in general ``no" -- the condition would imply that $K_X$ is a big divisor. Nevertheless, we can settle for something weaker, but still sufficient for our purposes: we require that the divisor $K_Y$ is \emph{semiample}. This then still produces an Iitaka fibration $g\colon Y\to Z$ and an ample divisor $A$ such that $K_Y=g^*A$, and the composite map $X\dashrightarrow Z$, which is now not necessarily birational, gives an isomorphism of section rings $R(X,K_X)$ and $R(Z,A)$.

Historically, by the influence of the classification of surfaces on the way we think about higher dimensional classification, this splits into two problems: finding a birational map $f\colon X\dashrightarrow Y$ such that the divisor $K_Y=f_*K_X$ is nef; and then proving that the nef divisor $K_Y$ is semiample. This last part -- the \emph{Abundance conjecture} -- is one of main open problems in higher dimensional geometry, in dimensions at least $4$. We know it holds in dimensions up to $3$, and when the canonical divisor is big, but very little is known in general.

Thus, hopefully by now it is clear that the main classification criterion is whether the canonical divisor $K_X$ is nef. If $K_X$ is nef, we are done, at least with the first part of the programme above. Life gets much tougher, but also much more interesting when the answer is \emph{no}.
\subsection*{The Cone and Contraction theorems}
Indeed, let $\NEb(X)\subseteq N_1(X)_\R$ denote the closure of the cone spanned by effective curves; note that the nef cone $\Nef(X)$ is dual to $\NEb(X)$ by Nakai's criterion, with respect to the intersection pairing. Since $K_X$ is not nef, the hyperplane 
$$K_X^\perp=\{C\in N_1(X)_\R\mid K_X\cdot C=0\}\subseteq N_1(X)_\R$$
must cut the cone $\NEb(X)$ into two parts; let us denote the two pieces by $\NEb(X)_{K_X\geq0}$ and $\NEb(X)_{K_X<0}$. Then the celebrated Cone theorem of Mori tells that the negative part $\NEb(X)_{K_X<0}$ is \emph{locally rational polyhedral}. More precisely:

\begin{thm}\label{thm:cone}
Let $X$ be a smooth projective variety. Then there exist countably many extremal rays $R_i$ of the cone $\NEb(X)$ such that $K_X\cdot R_i<0$ and $$\NEb(X)=\NEb(X)_{K_X\geq0}+\sum R_i.$$

Moreover, for every ample $\Q$-divisor $H$ on $X$, there exist finitely many such rays $R_i'$ with 
$$\NEb(X)=\NEb(X)_{K_X+H\geq0}+\sum R_i'.$$ 
In particular, the rays $R_i$ are discrete in the half-space $\NEb(X)_{K_X<0}$.
\end{thm}

Note that in the theorem, the second statement implies the first, by letting $H\to0$. This is the standard formulation, and the proof can be found in any treatise of the subject. A suitable formulation for klt pairs was proved by Kawamata, Koll\'ar and others. 

There is an additional statement that we can contract any of the extremal rays $R_i$ -- this is the Contraction theorem of Kawamata and Shokurov.

\begin{thm}\label{thm:contraction}
With the notation from Theorem \ref{thm:cone}, fix any of the rays $R=R_i$. Then there exists a morphism with connected fibres 
$$\cont_R\colon X\to Y$$
to a normal projective variety $Y$ such that a curve is contracted by $\cont_R$ if and only if its class lies in $R$.
\end{thm}

The importance of the Contraction theorem is two-fold. First, it is clear that such a contraction has to be defined by a basepoint free divisor $L$ with $L\cdot R=0$; in general, it is very difficult to show the existence of a single non-trivial non-ample basepoint free divisor on a variety -- the conclusion that there are many of them is clearly astonishing. 

Second, we want to eventually end up with a variety on which the canonical divisor is nef, i.e.\ it has no extremal rays as above. We therefore hope that by contracting some of the rays we can make the situation better. We will see below that this is not necessarily the case, at least not immediately. However, I will argue that life indeed gets better, at least if we choose carefully \emph{which} rays to contract.

I next state the result which contains both the Cone and Contraction theorems. The new statement lives in $N^1(X)_\R$ and, by duality, involves the nef cone. This formulation has been known for a long time, and origins go back at least to \cite{Kaw88}. However, it has only recently been realised \cite[Theorem 4.2]{CoL10} that this statement is much easier to prove than Theorems \ref{thm:cone} and \ref{thm:contraction}, once we have right tools at hand.

\begin{thm}\label{thm:ConeRevisited}
Let $X$ be a smooth projective variety such that $K_X$ is not nef. 
Let $V$ be the \emph{visible boundary} of $\Nef(X)$ from the class $\kappa=[K_X]\in N^1 (X)_\R$:
$$V=\big\{\delta\in\partial\Nef(X)\mid [\kappa,\delta]\cap\Nef(X)=\{\delta\}\big\}.$$
Then:
\begin{enumerate}
\item every compact subset $F$ which belongs to the relative interior of $V$, is contained in a union of finitely many supporting rational hyperplanes,
\item every Cartier divisor on $X$ whose class belongs to the relative interior of $V$ is semiample.
\end{enumerate}
\end{thm}
\begin{proof}
The proof is almost by picture. Note first that since $K_X$ is not nef, the class $\kappa$ is not in $\Nef(X)$. The set $V$ is then precisely the points that $\kappa$ ``sees" on $\Nef(X)$.
\begin{figure}[htb]
\begin{center}
\includegraphics[width=0.55\textwidth]{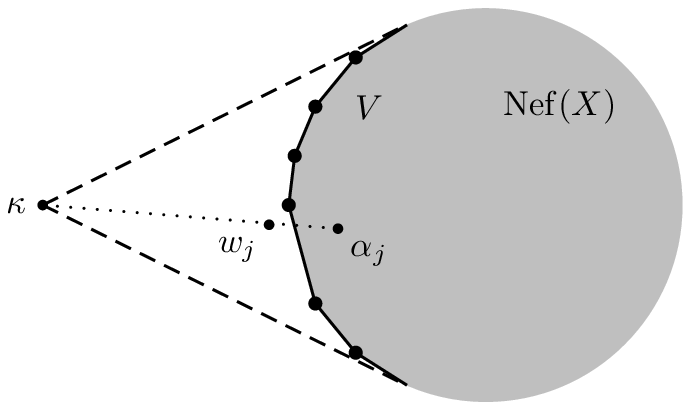}
\end{center}
\end{figure}

Since $F$ is compact, we can pick finitely many rational points $w_1,\dots,w_m\in N^1(X)_\R$ very close to $F$, such that $F$ is contained in the convex hull of these points. Then it is obvious that $F$ belongs to the boundary of the cone $\Nef(X)\cap\sum\R_+ w_i$, hence it is enough to show that this cone is rational polyhedral.

Note that since each $w_i$ is very close to $F$, and $F$ belongs to the relative interior of $V$, the line containing $\kappa$ and $w_i$ will intersect the ample cone. Therefore, there are rational ample classes $\alpha_j$ and rational numbers $t_j\in(0,1)$ such that 
$$w_j=t_j\kappa+(1-t_j)\alpha_j.$$ 
For each $j$, choose an ample $\Q$-divisor $A_j$ which represents the class $\frac{1-t_j}{t_j}\alpha_j$ such that the pair $(X,A_j)$ is klt (use Bertini's theorem). Then $w_j$ is the class of the divisor $t_j(K_X+A_j)$.  By Theorem~\ref{thmA}, the adjoint ring
\[
\mathfrak R=R(X;K_X+A_1, \dots, K_X+A_m)
\]
is finitely generated. Denote by $\pi\colon \Div_\R(X)\to N^1(X)_\R$ the natural projection. Then 
$$\Nef(X)\cap\sum\R_+w_i\subseteq\pi(\Supp\mathfrak R)$$
by Theorem \ref{thm:ELMNP}(2), and the conclusion follows by Proposition \ref{cor:5}.
\end{proof}

\begin{rem}\label{rem:cone}
Note that, with a bit more care, the proof above can be modified to prove the following: if additionally the divisor $K_X$ is big, then the whole set $V$ is contained in finitely many supporting rational hyperplanes. I leave the details to you.
\end{rem}

\begin{lem}\label{lem:implicationCone}
Let $X$ be a smooth projective variety such that $K_X$ is not nef. Then Theorem \ref{thm:ConeRevisited} (and Remark \ref{rem:cone}) imply the Cone and Contraction theorems.
\end{lem}
\begin{proof}
Note that the $K_X$-negative extremal rays are dual (with respect to the intersection pairing) to the rational hyperplanes containing faces of $\Nef(X)$ which are themselves contained in $V$. Indeed, as in the proof of Theorem \ref{thm:ConeRevisited}, any class $\delta$ in the relative interior of $V$ can be written as $\delta=t\kappa+(1-t)\alpha$ for some $t\in(0,1)$, where $\kappa$ is the class of $K_X$, and $\alpha$ is some ample class. Hence, any curve orthogonal to $\delta$ must be negative on $\kappa$. This immediately implies the lemma, since for any extremal ray $R$, the contraction $\cont_R$ is the Iitaka fibration of any line bundle which belongs to the interior of the face orthogonal to $R$.
\end{proof}

This is a good place to point out that the Cone theorem has an additional claim, which I deliberately chose not to include in the statement above. Namely, the rays $R_i$ are not generated by just any curves -- they are generated by \emph{rational} curves. Furthermore, we can choose a rational curve $C_i$ generating $R_i$ so that 
$$0<{-}K_X\cdot C_i\leq\dim X+1.$$ 
As we will see below, the existence of such rational curves is not necessary in order to perform the Minimal Model Program. However, it has very important structural consequences for the geometry of certain varieties, such as Fanos, which are covered by rational curves. 

This rational curves claim is the only part of the original Cone theorem which cannot be deduced from the proof in \cite{CoL10} presented above. The original Cone theorem is proved by an ingenious bend-and-break method of Mori which proceeds by reduction to positive characteristic. The ultimate dream is that finite generation techniques could provide an insight into how to prove statements about rational curves without passing to positive characteristic, which would then give hope that similar claims hold on a wider class of spaces, such as K\"ahler manifolds.
\subsection*{Contractions in the MMP}
Let us go back to the procedure in the Minimal Model Program. The Cone and Contraction theorems tell us that that if we pick a $K_X$-negative extremal ray $R$, we can contract it to obtain another normal projective variety $Y$, and we hope that it shares many of the properties of $X$ that we started with, for instance $\Q$-factoriality. Note that, by Theorem \ref{thm:ConeRevisited}, the map $\cont_R$ is given by any semiample divisor lying in the relative interior of the set $R^\perp\cap\Nef(X)$, and it is therefore birational (recall that we assumed at the beginning that $K_X$ is pseudo-effective). Here the situation branches into two distinct cases.

Assume first that the exceptional set of the map $\cont_R$ contains a prime divisor $E$. Then, in fact, we have $\Exc(\cont_R)=E$, and moreover, $Y$ is also $\Q$-factorial -- I will prove a more general version of this in Theorem \ref{lem:nullflip}, but one should note that the proof is almost identical. In this case, we say that $\cont_R$ is a \emph{divisorial} contraction. 

A drawback is that $Y$ is no longer necessarily smooth, but still it has singularities which are very close to the smooth case, and we can continue our programme on $Y$. However, something changed for the better: the Picard number dropped by $1$ since we contracted the divisor $E$; our variety became simpler. 

Assume next that the exceptional set of the map $\cont_R$ does not contain a prime divisor, i.e.\ that we have $\codim_X\Exc(\cont_R)\geq2$. In this case, we say that $\cont_R$ is a \emph{flipping} contraction. 

This situation is bad: not only do we have that $Y$ is not $\Q$-factorial, but even $K_Y=(\cont_R)_*K_X$ is not a $\Q$-Cartier divisor. Indeed, since $\cont_R$ is an isomorphism in codimension $1$, we have $K_X=\cont_R^*K_Y$. If $C$ is a curve contracted by $\cont_R$, then $K_X\cdot C<0$, and by the projection formula this equals $K_Y\cdot (\cont_R)_*C=0$, a contradiction.

The great insight of Mori, Reid and others is this. Note that the divisor $K_X$ is anti-ample with respect to the map $\cont_R$, and the result that we want to end up with in the end should give the canonical divisor which is nef. Thus, it is a natural thing to try to construct at least a birational map $X^+\to Y$ which ``turns the sign" of all curves contracted by $\cont_R$; in other words, it ``flips" them. Therefore, we would like to have a diagram:
\[
\xymatrix{ 
X \ar@{-->}[rr]^{\varphi} \ar[dr]_{\cont_R} & \quad & X^+
  \ar[dl]^{\cont_R^+}\\
\quad & Y & \quad
}\]
such that $K_{X^+}$ is ample with respect to $\cont_R^+$.

This diagram, or just the map $\varphi$, is called \emph{the flip} of $\cont_R$. Since, by our requirements, the map $\varphi$ should not extract divisors, the morphism $\cont_R^+$ is also an isomorphism in codimension $1$. It is then not too difficult, but crucial, to show that the existence of the diagram is equivalent to the fact that the relative canonical ring
$$R(X/Y,K_X)=\bigoplus_{n\in\N}(\cont_R)_*\OO_X(nK_X)$$
is finitely generated as a sheaf of algebras over $(\cont_R)_*\OO_X=\OO_Y$. It immediately follows that $X^+$ is $\Q$-factorial and that the Picard number of $X^+$ is the same as that of $X$.

The flip as above is by now proved to exist in any dimension. The first proof for threefolds was given by Mori in \cite{Mor88}. It was proved in general in \cite{BCHM} by MMP techniques, and in \cite{Laz09,CaL10} as a consequence of Theorem \ref{thmA}. 
\subsection*{Termination of the MMP}
The variety $X^+$, thus, has all the desired features similar to $X$, so we continue the procedure with $X^+$ instead of $X$ (again, as in the case of divisorial contractions, we lose smoothness, but we are all right if we slightly enlarge our category). Unfortunately, it is not easy to find an invariant of varieties which behaves well under flips; the only such example currently exists on threefolds. It is, therefore, the crucial problem to find a sequence of divisorial contractions and flips which terminates.

We know how to do this for varieties of general type, and this was proved first in \cite{BCHM}. Here, I give an argument close to that from \cite{CoL10} -- I hope to convince you that it is not too difficult to deduce it as a consequence of Theorem \ref{thmA}.

\begin{thm}\label{thm:termination}
Let $X$ be a variety of general type. Then there exists a sequence of $K_X$-divisorial contractions and $K_X$-flips which terminates.
\end{thm}
\begin{proof}[Sketch of the proof]
The proof is by double induction: the first level of induction is on the Picard number $\rho=\dim N^1(X)_\R$, and we can assume that the result holds for varieties with Picard number smaller than $\rho$. I define the second level of induction a few lines below.

Denote by $\pi\colon \Div_\R(X)\to N^1(X)_\R$ the natural projection. Similarly as in the proof of Theorem \ref{thm:ConeRevisited}, we choose ample $\Q$-divisors $A_1,\dots,A_m$ such that all the pairs $(X,A_i)$ are klt, such that the cone $\pi\big(\sum\R_+(K_X+A_i)\big)$ has dimension $\rho$, and that this cone contains an ample divisor. 

We can assume that $K_X$ is an effective divisor, and note that for $0<\varepsilon\ll1$, the pair $(X,\Delta=\varepsilon K_X)$ is klt. Therefore, by Theorem \ref{thmA}, the ring
$$\mathfrak R=R(X;K_X+\Delta, K_X+A_1, \dots ,K_X+A_m)$$
is finitely generated. We note that the cone 
$$\mcal C=\R_+K_X+\sum\R_+(K_X+A_i)$$
is equal to the support of $\mathfrak R$.

Let $\mcal C= \bigcup_{i\in I}\mcal C_i$ be the rational polyhedral decomposition as in Theorem \ref{thm:ELMNP}. The second level of induction is on the cardinality of the set $I$.

By Proposition \ref{cor:5}, the cone $\mcal C\cap \pi^{-1}\bigl(\Nef (X)\bigr)$ is rational polyhedral, and let $\mcal F$ be a codimension $1$ face of this cone which intersects the interior of $\mcal C$. If $R\subseteq N_1(X)_\R$ is the extremal ray of $\NEb(X)$ orthogonal to $\mcal F$, then the corresponding contraction $\cont_R$ is given by any basepoint free divisor which belongs to the interior of $\mcal F$, cf.\ the proof of Lemma \ref{lem:implicationCone}. 

If $\cont_R$ is divisorial, then we finish by induction on $\rho$. Therefore, we may assume that $\cont_R$ is flipping, and then by the discussion above, there exists the flip $\varphi\colon X\dashrightarrow Y$ of $\cont_R$. 

The map $\varphi$ is an isomorphism in codimension $1$, hence it induces isomorphisms $\Div_\R(X)\simeq \Div_\R(Y)$ and
$$\mathfrak R \simeq R(Y;K_Y+\varphi_*\Delta,K_Y+\varphi_*A_1,\dots,K_Y+\varphi_*A_m).$$
The cone $\mcal C'=\varphi_*\mcal C\subseteq \Div_\R(Y)$ has a decomposition $\mcal C'=\bigcup_{i\in I'}\mcal C_i'$ as in Theorem \ref{thm:ELMNP}, and it is a key step to show that we can assume that $I=I'$ and $\mcal C_i'=\varphi_*\mcal C_i$. In other words, it suffices to prove that if an asymptotic order function $o_\Gamma$ is linear on a cone $\mcal C_i$, then it is also linear on $\varphi_*\mcal C_i$. I omit the proof of this, but once one knows the correct statement, the proof becomes easy, see \cite[Lemma 5.2]{CoL10}.

It can be easily shown that for every $L\in\mcal F$, the divisor $\varphi_*L\in\Div_\R(Y)$ is again nef, but not ample. In other words, the set $\varphi_*\mcal F$ belongs to the boundary of the cone $\Nef(Y)$. Note that the interiors of the cones $\varphi_*\Nef(X)$ and $\Nef(Y)$ do not intersect, since otherwise $\varphi$ would be an isomorphism.

Let $V\subseteq\Div_\R(X)$ be the minimal vector space containing $\mcal C$. Let $\mathcal H\subseteq V$ be the rational hyperplane which contains $\mcal F$, and let $\mathcal C_{\ell}$, for $\ell\in J\subsetneq I$, be the cones such that $\mathcal C_{\ell}$ and $\Nef(X)$ are not on the same side of $\mathcal H$. 

If we pick rational generators $D_1,\dots,D_r$ of the cone $\mcal D=\bigcup_{\ell\in J}\mcal C_\ell'$, then the ring 
$$\mathfrak R'=R(Y;D_1,\dots,D_r)\simeq R(X;\varphi_*^{-1}D_1,\dots,\varphi_*^{-1}D_r)$$ 
is finitely generated by Lemma \ref{lem:3}. Also, we have $\Supp\mathfrak R'=\mcal D$, and this cone contains an ample divisor by the argument above. Since the size of $J$ is strictly smaller than that of $I$, we finish the proof.
\end{proof}

Note that the proof gives more -- it shows that the resulting minimal model for $K_X$ is at the same time a minimal model for \emph{every} divisor $D$ which is in the same chamber $\mcal C_{j_0}$ as $K_X$; note that $D$ is a multiple of an adjoint divisor, so it makes sense to talk about its minimal models. This is one of the guiding lights in this paper, so let us state it as a standalone result. It is, at this moment, convenient to switch to pairs and state it in greater generality.

\begin{thm}\label{thm:finiteness}
Let $X$ be a $\Q$-factorial projective variety, and let $\Delta_1,\dots,\Delta_r$ be big $\Q$-divisors such that each $(X,\Delta_i)$ is a klt pair. Let $\mcal C=\sum_{i=1}^r\R_+(K_X+\Delta_i)$. 

Then there exists a finite rational polyhedral subdivision $\mcal C=\bigcup\mcal C_k$ and finitely many birational maps $\varphi_k\colon X\dashrightarrow X_k$ such that $X_k$ is a minimal model for every divisor in $\mcal C_k$.
\end{thm}
\section{Picture 2: Mori Dream Spaces}\label{sec:picture2}

In their influential paper \cite{HK00}, Hu and Keel introduced a class of varieties called \emph{Mori Dream Spaces}, which have exceptionally nice birational properties. As we will see, they have only finitely many birational maps to other $\Q$-factorial varieties which are isomorphisms in codimension $1$, and they possess, in some sense, a canonically given finitely generated divisorial ring. 

The original motivation for their study lies in the theory of moduli spaces of curves; however, it was immediately realised that another big family of varieties -- toric varieties -- belongs to this class. Also, it was an expectation based on the Minimal Model Program that Fano varieties are Mori Dream Spaces, which was confirmed in \cite{BCHM}; I discuss this below.

The construction of Hu and Keel uses the theory of variation of GIT structures. I do not touch upon this beautiful theory here; instead, I try to convince you that there are obvious parallels between the ingredients and the output of the classical MMP on the one hand, and the theory of Mori Dream Spaces on the other. Then in the next section I argue that both of these are just instances of a more general theory.

Let us start with a definition of Mori Dream Spaces from \cite{HK00}; for the definition of the pullback of a divisor under a birational map, see for instance \cite[1.0]{HK00}.

\begin{dfn}\label{dfn:25}
A $\Q$-factorial projective variety $X$ is a Mori Dream Space if:
\begin{enumerate}
\item $\Pic(X)_\Q =N^1(X)_\Q$,
\item $\Nef(X)$ is the affine hull of finitely many semiample line bundles, and
\item there are finitely many birational maps $f_i\colon X \dashrightarrow X_i$ to projective $\Q$-factorial varieties $X_i$ such that each $f_i$ is an isomorphism in codimension $1$, each $X_i$ satisfies (2), and $\overline{\Mov}(X)=\bigcup f^*_i\big(\Nef(X_i)\big)$.
\end{enumerate}
\end{dfn}

These spaces are really as nice as a variety can get: all possible cones inside $N^1(X)_\R$ are rational polyhedral, and as mentioned above, the birational geometry is as simple as one can generally expect: if one defines a more general version of the Minimal Model Program -- as we will do in the following section -- it becomes clear that the maps $f_i$ above are just maps in that MMP.

There is an obvious parallel between Definition \ref{dfn:25} and the conclusion of Theorem \ref{thm:finiteness}. In both cases we have a certain distinguished cone with its distinguished finite rational polyhedral subdivision, which gives a \emph{variation (or geography) of minimal models}. On a random variety you cannot hope to do any better than this.

It is then a natural question to ask whether there exists a divisorial ring associated to a Mori Dream Space which is finitely generated, in analogy with Theorem \ref{thmA}. We will shortly see that this is indeed the case.

To this end, let $X$ be a Mori Dream Space, and let $D_1, \dots, D_r$ be a basis of $\Pic(X)_\Q$ such that $\Effb(X)\subseteq \sum \R_+D_i$. Then a \emph{Cox ring} of $X$ is 
\[
R(X;D_1,\dots,D_r)= \bigoplus_{(n_1,\dots,n_r)\in\N^r}H^0(X,n_1D_1+\dots+n_r D_r).
\]
This ring depends on the choice of divisors $D_1, \dots, D_r$, but the only thing we really care about is its finite generation, and that question is independent of the choice of $D_i$, cf.\ Lemma \ref{lem:3}. Then we have:

\begin{thm}\label{thm:MDSfingen}
Let $X$ be a Mori Dream Space. Then any of its Cox rings is finitely generated.
\end{thm}
\begin{proof}[Sketch of the proof]
I use the notation from Definition \ref{dfn:25}. The cone $\overline{\Mov}(X)$ is rational polyhedral by the definition of a Mori Dream Space, and let $\Q$-divisors $M_1,\dots,M_p$ be its generators.  Then the divisorial ring $R(X;M_1,\dots,M_p)$ is finitely generated: indeed, for each $i$, let $N_{i1},\dots,N_{ip_i}$ be rational generators of the cone $\Nef(X_i)$. Then all the divisors $f_i^*N_{ik}$ form a set of generators of the cone $\overline{\Mov}(X)$, and each of the rings 
$$R(X;f_i^*N_{i1},\dots,f_i^*N_{ip_i})\simeq R(X_i;N_{i1},\dots,N_{ip_i})$$ 
is finitely generated since all $N_{ik}$ are semiample. The conclusion follows by Lemma \ref{lem:3}.

Let $\mathcal F_\lambda$ be all the codimension $1$ faces of all $f_i^*\Nef(X_i)$, with the property that $\mcal F_\lambda$ belong to the boundary of the cone $\overline{\Mov}(X)$, and that each $\mcal F_\lambda$ contains a big divisor. Let $\vphi_\lambda\colon X\dashto X_\lambda$ be the birational contraction associated to $\mcal F_\lambda$ as in Theorem \ref{thm:decomposition}, and let $E_{\lambda k}$ be the exceptional divisors of $\vphi_\lambda$. Then each set $$\mcal D_\lambda=\mcal F_\lambda+\sum\nolimits_k\R_+ E_{\lambda k}$$ 
is a rational polyhedral cone. 

Now, it can be shown that 
$$\overline{\Eff}(X)=\overline{\Mov}(X)\cup\bigcup\nolimits_\lambda\mcal D_\lambda.$$ 
The idea is to ``project" the pseudo-effective cone onto the movable cone. Details are easy, but a bit tedious, see the proof of \cite[Corollary 4.4]{KKL12}. This implies, in particular, that the cone $\overline{\Eff}(X)$ is rational polyhedral and equal to $\Eff(X)$, and it suffices to show that each of the rings $R(X;G_{\lambda 1},\dots,G_{\lambda p_\lambda})$ is finitely generated, where $G_{\lambda 1},\dots,G_{\lambda p_\lambda}$ are generators of $\mcal D_\lambda$ (indeed, it is easy to see that generators of all these rings, together with generators of the ring $R(X;M_1,\dots,M_p)$, generate a suitable Cox ring of $X$). But this follows similarly as above -- the point is that the divisors $E_{\lambda k}$ do not add any sections as they are $\varphi_\lambda$-exceptional. The details are left to the reader.
\end{proof}

The converse of Theorem \ref{thm:MDSfingen} also holds, but we have to push our techniques a bit further -- this will be done in the next section. The following is a direct consequence of Theorem \ref{thm:scalingbig}.

\begin{thm}\label{thm:MDScharacterisation}
Let $X$ be a $\Q$-factorial projective variety such that $\Pic(X)_\Q=N^1(X)_\Q$. Let $D_1, \dots, D_r$ be a basis of $\Pic(X)_\Q$ such that $\Effb(X)\subseteq \sum \R_+D_i$. 

If the ring $R(X;D_1,\dots,D_r)$ is finitely generated, then $X$ is a Mori Dream Space.
\end{thm}

\begin{cor}\label{cor:MDS}
If $X$ is a Fano variety, then $X$ is a Mori Dream Space.
\end{cor}
\begin{proof}
By Kodaira vanishing, we have $H^i(X, \OO_X)=H^i(X,K_X+(-K_X))= 0$ for all $i>0$. The exponential sequence
\[
0 \longrightarrow \Z \longrightarrow \OO_X \longrightarrow \OO^*_X \longrightarrow 0
\]
then yields the exact sequence
$$0=H^1(X,\OO_X) \to H^1(X,\OO^*_X)\to H^2(X,\Z)\to H^2(X,\OO_X)=0,$$
and in particular $\Pic(X)_\Q= N^1(X)_{\Q}$. Let $D_1,\dots,D_r$ be a basis of $\Pic(X)_\Q$ such that $\overline{\Eff}(X)\subseteq \sum \R_+D_i$, and such that $A_i=D_i-K_X$ is ample for every $i$. Then the Cox ring 
$$R(X; D_1, \dots,D_r)=R(X; K_X+A_1, \dots, K_X+A_r)$$ 
is finitely generated by Theorem \ref{thmA}. We conclude by Theorem \ref{thm:MDScharacterisation}.
\end{proof}

It is worth mentioning that Mori Dream Spaces have connections to various important developments in Algebraic Geometry and beyond. I advise you to read the wonderful survey \cite{McK10} in order to get a flavour of some of these directions of research.
\section{MMP beyond the canonical class}\label{sec:mmpbeyond}

In this section I discuss what is understood by a good birational theory of an algebraic variety. When the variety in question is smooth or has mild singularities, this reduces to the classical Minimal Model Program associated to the canonical class. Another instance of this story is Mori Dream Spaces. Everything that happens in this section should be looked at through the prism of these two main examples, and the goal is to find \emph{the largest possible class} where we can run something that looks like the classical MMP.

It is a reasonable question whether we can always achieve the MMP for the canonical class. In Section \ref{sec:picture1} we saw that, at least conjecturally, we are able to perform the programme for the class of varieties which have mild singularities, say klt. In general, if we have a $\Q$-factorial projective variety $X$ with arbitrary singularities, one approach is to take a resolution $f\colon Y\to X$. We can write 
$$K_Y+\Gamma=f^*K_X+E,$$
where $\Gamma$ and $E$ are effective $\Q$-divisors with no common components, and $E$ is $f$-exceptional. Then one can either try to do the MMP for the canonical class $K_Y$, or try to do the MMP for $K_Y+\Gamma$. 

There are problems with both of these approaches. In the first one, it is expected that the MMP will terminate; however, the resulting $K_Y$-MMP will not have any of the properties that we would like it to have -- sections of $K_X$ will not be preserved, and we possibly extracted divisors along the way (by taking this first blowup $f$). The second case is in some sense even worse -- our MMP will, in general, not even terminate, since the canonical ring $R(Y,K_Y+\Gamma)\simeq R(X,K_X)$ might not be finitely generated, see Example \ref{e_canonical}. 

All this really comes up in nature, and the mild singularities involved in the Minimal Model Program are indeed necessary. The following example demonstrates this point clearly; it was kindly communicated to me by J.\ Koll\'ar, and it is a straightforward generalisation of Sakai's example \cite{Som86}.

\begin{exa}\label{e_canonical}
Let $Y$ and $M$ be as in Example \ref{e_cutkosky}. Set $L=M\otimes\omega_Y^{-1}\otimes\OO_Y(1)$ and $\mathcal E=L\oplus \OO_Y(1)^{\oplus 3}$, and let $Z=\mathbb P(\mathcal E)$ with the projection map $\pi\colon Z\to Y$. Thus, $Z$ is a smooth $\PS^3$-bundle over $Y$, and denote $\xi=\OO_Z(1)$. Then 
$$\omega_Z=\pi^*(\omega_Y\otimes\det\mathcal E)\otimes\xi^{\otimes-4}=\pi^*(\omega_Y\otimes L\otimes\OO_Y(3))\otimes\xi^{\otimes-4}.$$
Consider the linear system $|\xi \otimes\pi^*\OO_Y(-1)|$. It contains smooth divisors $S_1,S_2,S_3$ corresponding to the quotients $\mathcal E \to L\oplus \OO_Y(1)^{\oplus 2}$, and note that $P=S_1\cap S_2\cap S_3$ is a codimension $3$ cycle corresponding to the quotient $\mathcal E\to L$. In particular, the base locus of $|(\xi\otimes\pi^*\OO_Y(-1))^{\otimes4}|$ is contained in $P$. 

Let $X$ be a general member of $|(\xi\otimes\pi^*\OO_Y(-1))^{\otimes4}|$. Then $X$ is smooth in codimension $1$, and since $Z$ is smooth, we have that $X$ is normal and Gorenstein. The adjunction formula \cite[Proposition 16.4]{Kol92} gives
$$\omega_X=\omega_Z\otimes\OO_Z(X)\otimes\OO_X=(\pi_{|X})^*(\omega_Y\otimes L\otimes\OO_Y(-1))=(\pi_{|X})^*M.$$
In particular, the canonical ring $R(X,\omega_X)\simeq R(Y,M)$ is not finitely generated, and it is easy to check that the singularities of $X$ (in the sense of the MMP) are very bad. A (more complicated) variation of this produces an example which is of general type, I leave the details to a particularly ambitious reader.
\end{exa}

\subsection*{Finding a right setup}
Therefore, there are indeed situations where the classical Minimal Model Program cannot work for the canonical class. On the other hand, Mori Dream Spaces show that there are varieties where we can do a version of the MMP for \emph{every} effective divisor. Of course, this is an exceptionally nice extreme, and we would like to find, in some sense the \emph{maximal} class of varieties where a version of this programme can be performed. Maybe it is too much to hope that there exists such a class which contains both the classical MMP and Mori Dream Spaces, since they can be, in some sense, unrelated or only loosely related. However, we will see that we can indeed build a theory which contains both of these \emph{pictures} as special instances.

Say we have a $\Q$-factorial projective variety $X$ and a $\Q$-divisor $D$ on $X$; note that here we allow $X$ to be \emph{arbitrarily} singular. In general, we do not a priori know much about the properties of the pair $(X,D)$, unless the divisor $D$ is in some well-known class, say if $D$ is semiample. Hence, ideally we would like to have a birational map $f\colon X\dashrightarrow Y$ to a $\Q$-factorial projective variety $Y$ such that the $\Q$-divisor $f_*D$ is semiample. We also want that the map $f$ induces the isomorphism between section rings $R(X,D)$ and $R(Y,f_*D)$. If $E_1,\dots,E_\ell$ are the prime divisors contracted by $f$, then this is achieved if, for instance,
\begin{equation}\label{eq:pullpush}
D=f^*f_*D+\sum r_i E_i
\end{equation} 
for some $r_i\geq0$, so we impose this condition as well. Such $Y$ is then called a \emph{minimal model of $D$}.

We first notice that, if an MMP as above can be performed for our $\Q$-divisor $D$, then $D$ cannot be \emph{isolated} in the N\'eron-Severi space $N^1(X)_\R$. The following lemma makes this more precise.

\begin{lem}
Let $X$ be a $\Q$-factorial projective variety, and let $D$ be a $\Q$-divisor on $X$. Assume that there exists an MMP for $D$ as explained above, and let $\pi\colon\Div_\R(X)\to N^1(X)_\R$ be the natural projection. 

Then there exist $\Q$-divisors $D_1,\dots,D_r$ such that
\begin{enumerate}
\item $D\in\sum\R_+D_i\subseteq\Div_\R(X)$,
\item $\dim\pi(\sum\R_+D_i)=\dim N^1(X)_\R$,
\item the ring $R(X;D_1,\dots,D_r)$ is finitely generated.
\end{enumerate} 
\end{lem}
\begin{proof}
We assume the notation as above. In particular, let $f\colon X\dashrightarrow Y$ be an MMP for $D$. Since $f_*D$ is semiample, there exist semiample $\Q$-divisors $G_1,\dots,G_m$ on $Y$ such that:
\begin{enumerate}
\item $f_*D\in\sum\R_+G_i\subseteq\Div_\R(Y)$,
\item the dimension of the image of the cone $\sum\R_+G_i$ in $N^1(Y)_\R$ is maximal, and
\item the ring $R(Y;G_1,\dots,G_m)$ is finitely generated.
\end{enumerate}
Indeed, we take $G_1=f_*D$, and we can pick $G_2,\dots,G_m$ to be ample.

Recall that $E_1,\dots,E_\ell$ are the $f$-exceptional prime divisors on $X$, cf.\ \eqref{eq:pullpush}. Now we define $D_1,\dots,D_r$, with $r=m+\ell$, as follows. Set 
$$D_i=f^*G_i$$
for $i=1,\dots,m$, and set 
$$D_{m+i}=f^*G_1+\lambda_i E_i$$
for $i=1,\dots,\ell$, where $\lambda_i=\ell r_i$. Then it is easy to see that (1) and (2) hold. It remains to show that the ring $R(X;D_1,\dots,D_r)$ is finitely generated.

For non-negative integers $k_1,\dots,k_r$, denote $D_{k_1,\dots,k_r}=\sum k_i D_i$, and note that 
$$\textstyle D_{k_1,\dots,k_r}=\sum_{i=1}^m f^*(k_i G_i)+\big(\sum_{i=m+1}^r k_i\big)f^*G_1+\sum_{i=m+1}^r k_i\lambda_i E_i.$$
This implies
$$\textstyle H^0(X,D_{k_1,\dots,k_r})=H^0\Big(X,\sum_{i=1}^m k_i D_i+\big(\sum_{i=m+1}^r k_i\big)D_1\Big),$$
and thus
$$R(X;D_1,\dots,D_r)\simeq R(X;D_1,\dots,D_m,D_1,\dots,D_1).$$
Now, this last ring is finitely generated by Lemma \ref{lem:3}, as the ring $$R(X;D_1,\dots,D_m)\simeq R(Y;G_1,\dots,G_m)$$
is finitely generated.
\end{proof}

Therefore, unless you have a finitely generated divisorial ring $\mathfrak R$ such that $D\in\Supp\mathfrak R$ which is \emph{full} (in the sense that the image of $\Supp\mathfrak R$ in $N^1(X)_\R$ is maximal dimensional), then you stand no chance of ever performing the Minimal Model Program for this $D$.

Now that we have the graded ring $\mathfrak R=R(X;D_1,\dots,D_r)$, we recall that the chamber decomposition from Theorem \ref{thm:ELMNP} gives us natural maps 
$$\varphi_i\colon X\dashto X_i$$
as in Theorem \ref{thm:decomposition}. This resembles strongly the situations in Sections \ref{sec:picture1} and \ref{sec:picture2}, but as we see immediately, it fails in two crucial ways.

First, recall that one of the requirements that we had was that the maps $\varphi_i$ factor into \emph{elementary} maps as defined earlier. In general (in particular, if the support of $\mathfrak R$ does not contain mobile divisors), it seems hopeless to expect such a factorisation. We would like to imitate the procedure in Section \ref{sec:picture1}: to even start the process, $\Supp\mathfrak R$ has to intersect the ample cone. Thus, we include the condition that $\Supp\mathfrak R$ contains an ample divisor -- in most applications, like in the context of Sections \ref{sec:picture1} and \ref{sec:picture2}, this is harmless.

Second, a fundamental requirement is that all $X_i$ are $\Q$-factorial varieties. Recall that $X_i=\Proj R(X,D_i)$ for some (equivalently, any) $\Q$-divisor $D_i$ in the interior of the chamber $\mcal C_i$. Let $D_i'$ be any $\Q$-divisor such that $D_i\equiv D_i'$. If $X_i$ is $\Q$-factorial, then in particular, the divisor $(\varphi_i)_*D_i'$ is $\Q$-Cartier. It is easy to show \cite[Lemma 4.6]{KKL12} that in that case, the section ring $R(X,D_i')$ is also finitely generated: one has to show that $(\varphi_i)_*D_i\equiv (\varphi_i)_*D_i'$ and that $\varphi_i$ preserves the section ring $R(X,D_i')$; since $(\varphi_i)_*D$ is ample, the conclusion follows.

Therefore, the divisors in the interior of $\Supp\mathfrak R$ must be pretty special -- it is not in general true that finite generation of section rings is a numerical property, see Example \ref{exa:notgen}. These divisors deserve a special name.

\begin{dfn}\label{dfn:gen}
Let $X$ be a $\Q$-factorial projective variety. A $\Q$-divisor $D$ is \emph{gen} if for every $\Q$-divisor $D'\equiv D$, the section ring $R(X,D')$ is finitely generated. 
\end{dfn}

Therefore, our last requirement must be that all the divisors in the interior of $\Supp\mathfrak R$ are gen. We are on the right track: it is an easy exercise to check that this is the situation in both Sections \ref{sec:picture1} and \ref{sec:picture2}.

To finish this part of the discussion, one can ask whether this is a redundant condition -- maybe it is true that if you have a divisorial ring $\mathfrak R$ as above (such that the dimension of $\Supp\mathfrak R$ is maximal, and this cone contains an ample divisor), then the interior of $\Supp\mathfrak R$ is made up of gen divisors automatically. The following example \cite[Example 4.8]{KKL12} shows that this is not the case.

\begin{exa}\label{exa:notgen}
Let $E$, $D$ and $Y$ be as in Example \ref{e_cutkosky}, and let $p \colon Y\to E$ be the natural projection. Note that $Y$ is a smooth surface with Picard number $2$. Consider 
$$L_1=  c_1\big(\OO_Y(1)\big)\quad\text{and}\quad L_2=c_1\big(\OO_Y(1)\otimes p^*\OO_{E}(-D)\big).$$ 
Then $L_1$ and $L_2$ are numerically equivalent nef and big divisors. One can show that $L_2$ is semiample while $L_1$ is not, and that $R(Y, L_2)$ is finitely generated while $R(Y, L_1)$ is not, see \cite[Example 10.3.3]{Laz04}.

Since $L_2$ is semiample but not ample, there exists an irreducible curve $C$ on $Y$ such that $L_2\cdot C=0$. Since $L_2$ is big and nef, we have $L_2^2>0$, so the Hodge index theorem then implies $C^2<0$. 

Now, set $F = L_2+C$, and note that $F$ is not nef. Let $H$ be any ample divisor on $Y$. Since the Picard number of $Y$ is $2$, it follows that $L_2\in\R_+F+\R_+H$. Set
$$\mathfrak R=R(Y;F,H).$$
Then one can show the following:
\begin{enumerate}
\item $\mathfrak R=R(Y;F,H)$ is finitely generated, and
\item \emph{none} of the divisors in the cone $\R_+F+\R_+L_2\subseteq\Supp\mathfrak R$ is gen. 
\end{enumerate}
For (2), it suffices to observe that any divisor in $\R_+F+\R_+L_2$ is numerically equivalent to a non-negative linear combination of $L_1$ and $C$. The details are an easy exercise.
\end{exa}

\subsection*{Existence of extremal contractions}
Let $X$ be a $\Q$-factorial projective variety and let $D_1,\dots,D_r$ be $\Q$-divisors on $X$ such that the ring $\mathfrak R=R(X;D_1,\dots,D_r)$ is finitely generated. Denote by $\pi\colon\Div_\R(X)\to N^1(X)_\R$ the natural projection. Assume that $\Supp\mathfrak R$ contains an ample divisor, that $\pi(\Supp\mathfrak R)$ spans $N^1(X)_\R$, and that every divisor in the interior of $\Supp\mathfrak R$ is gen. Fix a divisor $G\in\Supp\mathfrak R$. I will show that, following the general strategy of Section \ref{sec:picture1}, we can construct an MMP for $G$.

By Proposition \ref{cor:5}, the cone $\mcal N=\Supp\mathfrak R\cap\pi^{-1}\big(\Nef(X)\big)$ is rational polyhedral, and every element of this cone is semiample. We pick any of its codimension $1$ faces $\mcal F$, such that $\mathcal F$ intersects the interior of $\mcal N$, and that $G$ and $\mcal N$ are not on the same side of the rational hyperplane which contains $\mcal F$. Then any line bundle in the relative interior of $\mcal F$ gives a birational contraction $f\colon X\to Y$. 

As in the classical setting, there are two cases: either the codimension of the exceptional locus is $1$, or it is at least two. We will see in the following Theorem \ref{lem:nullflip} that the first case is good, and in the second case we need a bit more work to rectify the fact that $Y$ is not $\Q$-factorial; however, strategy in both cases is similar. Then we can finish similarly as in Theorem \ref{thm:termination}, the details are in \cite[Theorem 5.4]{KKL12}.

The main deal is, therefore, to prove the existence of elementary contractions. This is far from straightforward -- observe that the construction of flips in Section \ref{sec:picture1} was an immediate consequence of finite generation (more precisely, of the relative version of Theorem \ref{thmA}). However, unlike in the classical situation, we have only one finitely generated ring $\mathfrak R$ at hand, and therefore we cannot apply any of the standard techniques like restricting ourselves to open subsets and so on. 

Thus, the following is the main result \cite[Theorem 5.2]{KKL12}.

\begin{thm}\label{lem:nullflip}
Let the notation and assumptions be as above. Then:
\begin{enumerate}
\item if the exceptional locus of $f$ contains a divisor, then this locus is a single prime divisor, and $Y$ is $\Q$-factorial,
\item if $f$ is an isomorphism in codimension $1$, then there exists a diagram
\[
\xymatrix{ 
X \ar@{-->}[rr]^{\varphi} \ar[dr]_f & \quad & X^+ \ar[dl]^{f^+}\\
\quad & Y & \quad
}\]
such that $\varphi$ is an isomorphism in codimension $1$ which is not an isomorphism, and $X^+$ is $\Q$-factorial.
\end{enumerate}
\end{thm}

\begin{proof}
Let $\Supp\mathfrak R=\bigcup\mcal C_i$ be the decomposition as in Theorem \ref{thm:ELMNP}. By the proof of Proposition \ref{cor:5}, there is a cone $\mathcal C_j\nsubseteq\pi^{-1}\big(\Nef(X)\big)$ of dimension $\dim\Supp\mathfrak R$, such that $\mcal F$ is a face of $\mathcal C_j$. Let $\varphi\colon X\dashrightarrow X^+$ be the birational contraction associated to any line bundle in the interior of $\mcal C_j$. Then by Theorem \ref{thm:decomposition}, there is a morphism $f^+\colon X^+\longrightarrow Y$ and the diagram as above. 
\begin{figure}[htb]
\begin{center}
\includegraphics[width=0.4\textwidth]{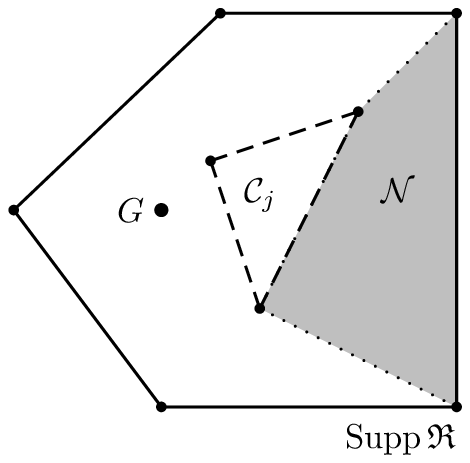}
\end{center}
\end{figure}

Note that every curve contracted by $f$ is orthogonal to every divisor in $\mcal F$, hence all such curves are proportional in $N_1(X)_\R$. Let $R\subseteq N^1(X)_\R$ denote the ray that they span. 

If $f$ is an isomorphism in codimension $1$, then so are $\vphi$ and $f^+$ as $\vphi$ is a contraction. 

If there exists an $f$-exceptional prime divisor $E$, then by the Negativity lemma \cite[Lemma 3.39]{KM98}, we have $E\cdot C<0$ for every curve $C$ contracted by $f$. Thus $C\subseteq E$, and so the exceptional locus of $f$ equals $E$. 

Next we have:

\begin{cla}\label{claim:pullback}
Let $M$ be a $\Q$-divisor on $X$ such that $M\equiv_f 0$. Then $M\sim_{\Q} f^*M_Y$ for some $\Q$-Cartier $\Q$-divisor $M_Y$ on $Y$. 
\end{cla}
This is the crucial part of the proof which uses gen divisors. Before proving the claim, let us show how it immediately implies the theorem.

Assume first that $f$ is an isomorphism in codimension $1$, and we need to show that $X^+$ is $\Q$-factorial. Consider a Weil divisor $P$ on $X^+$, and let $P'$ be its proper transform on $X$. Since $X$ is $\Q$-factorial, the divisor $P'$ is $\Q$-Cartier, and we pick any $\Q$-divisor $G \in \mcal{C}_j$. Since all the curves contracted by $f$ belong to $R$, there exists a rational number $\alpha$ such that $P'\equiv_f\alpha G$. By Claim \ref{claim:pullback}, there exists a $\Q$-Cartier $\Q$-divisor $D$ on $Y$ such that $P'\sim_\Q\alpha G+f^* D$. By the definition of $\vphi$, the divisor $f_*G$ is ample, hence $\Q$-Cartier. Therefore, pushing forward this relation by $\varphi$, we obtain that the divisor
$$P\sim_\Q \alpha f_*G+(f^+)^*D$$
is $\Q$-Cartier.

Now, if $f$ contracts a divisor $E$, we need to show that $Y$ is $\Q$-factorial. Let $P$ be any Weil divisor on $Y$, and let $P'$ be its proper transform on $X$. Then $P'$ is $\Q$-Cartier, and since all the curves contracted by $f$ belong to $R$, there exists a rational number $\alpha$ such that $P'\equiv_f \alpha E$. But then, by Claim \ref{claim:pullback} there exists a $\Q$-Cartier $\Q$-divisor $D$ on $Y$ such that $P'\sim_\Q\alpha E+f^*D$. Pushing forward this relation by $f$, we obtain that the divisor $P\sim_\Q D$ is $\Q$-Cartier.

It remains to prove Claim \ref{claim:pullback}. First, note that we can find $\Q$-divisors $B_1,\dots,B_r$ in the relative interior of $\mcal F$ such that $M\equiv \sum \lambda_i B_i$ for some nonzero rational numbers $\lambda_i$: indeed, by assumption, the set $\pi(\mcal F)$ spans the hyperplane in $N^1(X)_\R$ which is orthogonal to $R$, hence $M$ belongs to this hyperplane. Note that all $B_i$ are semiample by Proposition \ref{cor:5}. Denote 
$$B_1'=\frac{1}{\lambda_1}\textstyle\big(M-\sum_{i\geq2} \lambda_i B_i\big).$$
Then $B'_1\equiv B_1$, and observe that $B_1'$ is also semiample by Lemmas \ref{lem:numericalBig} and \ref{lem:ords} since $B_1$ is gen. Then by Lemma \ref{lem:equal proj}, there is an isomorphism 
$$\eta \colon Y\to \Proj R(X,B_1')$$
such that $f'= \eta\circ f$, where $f'\colon X\to \Proj R(X,B_1')$. By the definition of $f$, there are ample $\Q$-divisors $A_i$ on $Y$ such that $B_i\sim_\Q f^*A_i$ for all $i\geq 2$, and similarly, $B_1'\sim_\Q(f')^*A_1'$ for an ample $\Q$-divisor $A_1'$ on $\Proj R(X,B_1')$. Therefore $M\sim_\Q f^*M_Y$ for $M_Y=\lambda_1\eta^*A_1'+\sum_{i\geq2} \lambda_i A_i$.
\end{proof}

Finally, as in Section \ref{sec:picture1}, combining everything together yields the following, which is usually called the \emph{geography of (minimal) models}.

\begin{thm}\label{thm:scalingbig}
Let $X$ be a $\Q$-factorial projective variety, and let $D_1,\dots,D_r$ be $\Q$-divisors on $X$. Denote by $\pi\colon\Div_\R(X)\lto N^1(X)_\R$ the natural projection. Assume that the ring $R(X;D_1,\dots,D_r)$ is finitely generated, that $\Supp\mathfrak R$ contains an ample divisor, that $\dim\pi(\Supp\mathfrak R)=\dim N^1(X)_\R$, and that every divisor in the interior of $\Supp\mathfrak R$ is gen. 

Then there is a finite rational polyhedral decomposition
$$\Supp\mathfrak R= \bigcup\mcal{C}_i,$$
together with birational contractions $\vphi_i\colon X\dashto X_i$ to $\Q$-factorial projective varieties $X_i$, such that $X_i$ is a minimal model for every $D\in\mcal C_i$.
\end{thm}

\section{The Cone conjecture}\label{sec:CY}

My main goal in this section is to convince you that several important conjectures and theories are related: the Cone conjecture for Calabi-Yau manifolds, the Abundance conjecture, finiteness of minimal models up to isomorphisms, and the theory of Mori Dream Spaces. Taken together, we will see in Proposition \ref{pro:conjecturesimply} that they, at least morally, form a consistent picture.

Recall that according to the Minimal Model Program, starting with a variety $X$ with mild singularities on which $K_X$ is pseudo-effective, we expect that there exists a birational map $\varphi\colon X\dashrightarrow Y$ such that the divisor $K_Y$ is semiample. In particular, either $\kappa(X,K_X)=0$, or there exists a fibration $\theta\colon Y\to Z$ such that for the generic fibre $F$ we have $\kappa(F,K_F)=0$. Thus, when $\kappa(X,K_X)\geq1$, we can study the geometry of $Y$ via the geometry of the target $Z$ and that of the generic fibre $F$. Similarly, when $\kappa(X,K_X)={-}\infty$, we know that there exists a $K_X$-MMP which terminates with a variety $Y$ which has a Mori fibre space structure over a lower dimensional base $Z$; in particular, the general fibre of the map $Y\to Z$ is a Fano variety.

Therefore, conjecturally, the study of algebraic varieties splits into three distinct cases: when $K_X$ is either ample, anti-ample, or a torsion divisor. Much is known about the geometry (at least of moduli) in the first two cases. The third case, which I here call \emph{varieties of Calabi-Yau type}, form a rich and extensively studied class. 

\begin{rem}
Note that there are many definitions of a \emph{Calabi-Yau manifold}. Most often, a Calabi-Yau manifold is a smooth projective variety with $\omega_X\simeq\OO_X$ and $H^1(X,\OO_X)=0$. Sometimes it is required that additionally $X$ is simply connected and $H^i(X,\OO_X)=0$ for all $1\leq i\leq\dim X$.
\end{rem}

One of the basic questions in the classification of varieties is how many minimal models a (say, smooth) variety $X$ can have. The answer is known to be finite when $X$ is of general type by \cite{BCHM}. If $X$ is not of general type, there are known examples when this number is infinite. However, there is a conjecture that this number is \emph{finite up to isomorphism}, which means that we ignore birational identifications with $X$. 

Consider a variety $X$ of Calabi-Yau type, and denote by ${\Aut}(X)$ the automorphism group and by ${\Bir}(X)$ the group of birational automorphisms. Note that every element of $\Bir(X)$ is an automorphism in codimension $1$, which is an easy consequence of the Negativity lemma \cite[Lemma 3.39]{KM98}. We have a natural homomorphism
$$ r\colon \Bir(X) \to {\GL}(N^1(X)) $$
given by $g\mapsto g^*$. We set $\mathcal A(X) = r\big(\Aut(X)\big)$ and $\mathcal B(X) = r\big(\Bir(X)\big)$. 

\begin{rem}
In general, on a variety $X$ it is more convenient in our context below to consider the group $\PsAut(X)$ of pseudo-automorphisms acting on $N^1(X)$ instead of the group of birational isomorphisms $\Bir(X)$: here, elements of $\PsAut(X)$ are birational automorphisms which are isomorphisms in codimension $1$. Then we denote $\mcal P(X)=r\big(\PsAut(X)\big)$. 
\end{rem}

It is a basic question, interesting on its own, how $\Aut(X)$ and $\Bir(X)$, or equivalently $\mcal A(X)$ and $\mcal B(X)$, act on certain cones in $N^1(X)_\R$. The first thing to notice is that $\mcal B(X)$ preserves the effective cone $\Eff(X)$ and the movable cone $\overline{\Mov}(X)$, and that $\mcal A(X)$ preserves the nef cone $\Nef(X)$. 

A more precise answer is suggested by the following \emph{Cone conjecture}. But first we need a definition.

\begin{dfn}
Let $V$ be a real vector space equipped with a rational structure, and let $\mcal C$ be a cone in $V$. Let $\Gamma$ be a subgroup of $\GL(V)$ which preserves $\mcal C$. There a rational polyhedral cone $\Pi\subseteq\mcal C$ is a \emph{fundamental domain} for the action of $\Gamma$ on $\mcal C$ if the following holds:
\begin{enumerate}
\item $\mcal C=\bigcup_{g\in\Gamma}g\Pi$,
\item $\inte\Pi\cap \inte g\Pi=\emptyset$ if $g\neq\id$.
\end{enumerate}
\end{dfn}

\begin{con}
Let $X$ be a variety of Calabi-Yau type.
\begin{enumerate}
\item There exists a rational polyhedral cone $\Pi$ which is a fundamental domain for the action of $\mcal A(X)$ on $\Nef(X)\cap\Eff(X)$.
\item There exists a rational polyhedral cone $\Pi'$ which is a fundamental domain for the action of $\mcal B(X)$ on $\overline{\Mov}(X)\cap\Eff(X)$.
\end{enumerate}
\end{con}

The first part of the conjecture was formulated by Morrison \cite{Mor93} and was inspired by developments in mirror symmetry. It was extended to the statement about the movable cone by Kawamata \cite{Kaw97}, and there is a formulation which involves klt pairs and pseudo-automorphisms in Totaro's paper \cite{Tot10}. 

Kawamata also formulated the conjecture in the relative setting, i.e.\ when there exists a fibration $X\to S$ such that $K_X\equiv_S0$, in which case one should consider the groups $\Aut(X/S)$ and $\Bir(X/S)$, and the relative cones over $S$ instead. A positive answer to this form of the conjecture together with the Abundance conjecture would, in particular, give a positive answer to the conjecture stated above about finiteness of minimal models up to isomorphism. This gives the main motivation for the Cone conjecture.

Before I give some more motivation, let me briefly survey what is known (there are several papers which give a detailed history of the problem and the state of the art, see for instance \cite{Tot10b}). The starting point is the proof of the conjecture on Calabi-Yau surfaces by Sterk, Looijenga, Namikawa and Kawamata \cite{Ste85,Nam85,Kaw97}. This uses fully the global Torelli theorem for K3s. This was generalised by Totaro \cite{Tot10} to klt Calabi-Yau pairs -- the proof reinterprets the problem by using hyperbolic geometry. For abelian varieties, the proof is in \cite{PS10}. Finiteness of minimal models was proved for a class of holomorphic symplectic 4-folds in \cite{HT10}, and a version for the movable cone on projective holomorphic symplectic manifolds in \cite{Mar11}. Oguiso \cite{Og11} gave a proof of the conjecture for the movable cone of generic hypersurfaces of multi-degree $(2,\dots,2)$ in $(\PS^1)^n$ for $n\geq4$. Kawamata \cite{Kaw97} gave a proof of (a weaker form of) the relative version of the conjecture when $X\to S$ is a $3$-fold over a positive-dimensional base. This, in particular, showed that if $X$ is a $3$-fold with positive Kodaira dimension, then the number of its minimal models is finite up to isomorphisms. Finally, the conjecture was confirmed for Calabi-Yau $n$-folds with Picard number $2$ and infinite group $\Bir(X)$ in \cite{LP12}. 

\subsection*{Further motivation}
%
%
%

We start with the following result from convex geometry.

\begin{pro}\label{thm:Looijenga}
Let $V$ be a finite dimensional real vector space equipped with a rational structure, and let $L$ be a lattice in $V$. Let $\mcal C$ be a rational polyhedral cone in $V$ of dimension $\dim V$. Let $\Gamma$ be a subgroup of $\GL(V)$ which preserves $L$ and $\mcal C$. 

Then $\Gamma$ is a finite group, and there exists a rational polyhedral fundamental domain for the action of $\Gamma$ on $\mcal C$.
\end{pro}
\begin{proof}
Let $\delta_1,\dots,\delta_r$ be \emph{primitive classes} on the extremal rays of the cone $\mcal C$ (in the sense that they are integral classes not  divisible in $L$). Then any element $g\in\Gamma$ permutes these $\delta_i$: this follows since $g$ preserves $\mcal C$, and it sends a primitive class to a primitive class. Therefore, $\Gamma$ is finite.

The proof of existence of a rational polyhedral fundamental domain is a bit more involved. For every point $x\in V$, let $\Sigma_x$ denote the stabiliser of $x$ in $\Gamma$. Pick a point $x_0\in\mcal C$ such that for every $z\in\mcal C$ we have $|\Sigma_{x_0}|\leq|\Sigma_z|$. Then $\Sigma_{x_0}$ is actually trivial. Indeed, there exists $0<\varepsilon\ll1$ such that if $B(x_0,\varepsilon)$ is the $\varepsilon$-ball around $x_0$ (in the standard norm), then the sets $g\big(B(x_0,\varepsilon)\cap\mcal C\big)$ are pairwise disjoint for $g\notin\Sigma_{x_0}$. By the choice of $x_0$, this implies that $|\Sigma_{x_0}|=|\Sigma_z|$ for every $z\in B(x_0,\varepsilon)\cap\mcal C$. Hence, for every $g\in\Sigma_{x_0}$ we have that $g$ stabilises $B(x_0,\varepsilon)\cap\mcal C$, and thus $g=\id$ since there exists a basis of $V$ which belongs to $B(x_0,\varepsilon)\cap\mcal C$.

If $\langle\,\,,\,\rangle$ denotes the standard scalar product on $V\simeq\R^N$, for every $x,y\in V$ set
$$d(x,y)=\sum_{g\in\Gamma}\langle gx,gy\rangle.$$
Then it is easy to check that $d\colon V\times V\to\R$ is a scalar product, and that $d(x,y)=d(gx,gy)$ for every $x,y\in V$ and every $g\in\Gamma$. Let
$$\Pi=\{x\in\mcal C\mid d(x,x_0)\leq d(x,gx_0)\text{ for every }g\in\Gamma\}.$$
Then $\Pi$ is cut out from $\mcal C$ by rational half-spaces, and hence $\Pi$ is a rational polyhedral cone. I claim that $\Pi$ is a fundamental domain for the action of $\Gamma$ on $\mcal C$. Indeed, take any $w\in\mcal C$. Then there exists $h\in\Gamma$ such that $d(w,hx_0)\leq d(w,gx_0)$ for every $g\in\Gamma$. This is equivalent to 
$$d(h^{-1}w,x_0)\leq d(h^{-1}w,h^{-1}gx_0)$$
for every $g\in\Gamma$, and hence $h^{-1}w\in\Pi$. Therefore, $\mcal C=\bigcup_{g\in\Gamma}g\Pi$. Since $\Sigma_{x_0}=\{\id\}$, we have $\inte\Pi\cap\inte g\Pi=\emptyset$ unless $g=\id$ by definition of $\Pi$. This completes the proof.
\end{proof}

In our situation, $V$ is the N\'eron-Severi space $N^1(X)_\R$ with the standard lattice $L$ given by the N\'eron-Severi group $N^1(X)$ and the induced rational structure. Then the immediate consequence is the following.

\begin{cor}
Let $X$ be a projective variety such that the cone $\Nef(X)\cap\Eff(X)$ is rational polyhedral. Then the group $\mcal A(X)$ is finite, and there exists a fundamental domain for the action of $\mcal A(X)$ on $\Nef(X)\cap\Eff(X)$. 

Similarly, if the cone $\overline{\Mov}(X)\cap\Eff(X)$ is rational polyhedral, then the group $\mcal P(X)$ is finite, and there exists a fundamental domain for the action of $\mcal P(X)$ on $\overline{\Mov}(X)\cap\Eff(X)$. 

In particular, the conclusions above hold if $X$ is a Mori Dream Space.
\end{cor}
\begin{proof}
This is straightforward from Proposition \ref{thm:Looijenga}. As a side remark, note that if $\Nef(X)\cap\Eff(X)$ is a rational polyhedral cone, then this cone is equal to its closure, which must be $\Nef(X)$; similarly, we have $\overline{\Mov}(X)\cap\Eff(X)=\overline{\Mov}(X)$ if $\overline{\Mov}(X)\cap\Eff(X)$ is a rational polyhedral cone. 
\end{proof}

As a special case of the previous result, we have that ``on a Fano manifold the Cone conjecture holds''. As Totaro points out in \cite{Tot10b}, we can think of Calabi-Yau manifolds as ``just beyond" Fanos. Of course, Calabi-Yaus behave less well than Fanos: for instance, it is not too difficult to construct examples of Calabi-Yaus for which the nef or the movable cone are not rational polyhedral; one such convenient example is in Example \ref{ex:Oguiso}. However, the Cone conjecture gives a description of these cones which is the best that we can ever hope for, which is one additional motivation for it, at least philosophically. 

\begin{exa}\label{ex:Oguiso}
This is \cite[Proposition 1.4]{Og12}. Let $X$ be the intersection of general hypersurfaces in $\PS^3\times\PS^3$ of bi-degrees $(1,1)$, $(1,1)$, and $(2,2)$. Then $X$ is a Calabi-Yau $3$-fold $X$ of Picard number $2$, and it has the following properties: the boundary rays of the pseudo-effective cone (which, in this case, is the same as the movable cone) are both irrational, and the group $\Bir(X)$ is an infinite group.
\end{exa}

Let us see how much of the properties of Fanos we can recover on Calabi-Yaus. The start for our discussion is the following.

\begin{thm}\label{cor:CY finitely generated}
Let $X$ be a variety of Calabi-Yau type, and let $B_1,\dots,B_q$ be big $\Q$-divisors on $X$. Then the ring $R(X;B_1,\dots,B_q)$ is finitely generated. 
\end{thm}
\begin{proof}
We can assume that each $B_i$ is effective. Let $\varepsilon>0$ be a rational number such that all pairs $(X,\varepsilon B_i)$ are klt, and notice that $K_X+\varepsilon B_i\sim_\Q\varepsilon B_i$. Hence, the ring
$$R(X;B_1,\dots,B_q)$$ 
is finitely generated by Theorem \ref{thmA} and by Lemma \ref{lem:3}.
\end{proof}

We immediately obtain the following.

\begin{thm}\label{thm:CY}
Let $X$ be a variety of Calabi-Yau type.
\begin{enumerate}
\item The cone $\Nef(X)\cap\B(X)$ is locally rational polyhedral in $\B(X)$, and every element of $\Nef(X)\cap\B(X)$ is semiample.
\item The cone $\overline{\Mov}(X)\cap\B(X)$ is locally rational polyhedral in $\B(X)$.
\end{enumerate} 
\end{thm}
\begin{proof}
Part (1) was first proved in \cite[Theorem 5.7]{Kaw88}. The problem of finding the shape of $\overline{\Mov}(X)\cap\B(X)$ was posed in \cite[Problem 5.10]{Kaw88}. This was solved in \cite[Corollary 2.7]{Kaw97} for $3$-folds, and in \cite[Theorem 3.8]{KKL12} in general.

Let $V$ be a relatively compact subset of the boundary of $\overline{\Nef}(X)\cap\B(X)$, and denote by $\pi\colon \Div_\R(X)\lto N^1(X)_\R$ the natural projection. Then we can choose finitely many big $\mathbb Q$-divisors $B_1,\dots,B_q$ such
that $V\subseteq\pi(\sum_{i=1}^q\mathbb R_+B_i)$. Theorem~\ref{cor:CY finitely generated} implies that the ring  $\mathfrak R=R(X;B_1,\dots,B_q)$ is finitely generated, and hence $\pi^{-1}\big(\overline{\Nef}(X)\big)\cap \Supp\mathfrak R$ is a rational polyhedral cone and its every element is semiample by Corollary \ref{cor:5}. But then $V$ is contained in finitely many rational hyperplanes. This shows (1), and the proof of (2) is similar.
\end{proof}

Let us recall the following known conjecture which generalises Theorem \ref{thmA}.

\begin{con}\label{con:finitegeneration}
Let $X$ be a $\Q$-factorial projective variety, and let $\Delta_1,\dots,\Delta_r$ be $\Q$-divisors such that all pairs $(X,\Delta_i)$ are klt.

Then the adjoint ring
\[
R(X;K_X+\Delta_1,\dots,K_X+\Delta_r)
\]
is finitely generated.
\end{con}

This conjecture is implied by the full force of the MMP \cite{CS11}, including termination of any sequence of flips and Abundance, although a priori it is weaker than the MMP. In particular, the conjecture is a theorem in dimensions up to $3$.

Finally, the following result shows that Conjecture \ref{con:finitegeneration} and the Cone conjecture are, in some sense, consistent.

\begin{pro}\label{pro:conjecturesimply}
Let $X$ be an $n$-dimensional variety of Calabi-Yau type. Assume either Conjecture \ref{con:finitegeneration} in dimension $n$, or the Cone conjecture in dimension $n$. 

Then the cones $\Nef(X)\cap\Eff(X)$ and $\overline{\Mov}(X)\cap\Eff(X)$ are spanned by rational divisors.
\end{pro}
\begin{proof}
I only show the statements for $\Nef(X)\cap\Eff(X)$, the rest is analogous.

Assume Conjecture \ref{con:finitegeneration} in dimension $n$. Let $D$ be an $\R$-divisor whose class is in $\Nef(X)\cap\Eff(X)$. Then we can write $D\equiv\sum_{i=1}^r\delta_i D_i$ for prime divisors $D_i$ and positive real numbers $\delta_i$. Fix an ample $\Q$-divisor $A$ on $X$. By Theorem \ref{cor:CY finitely generated}, the ring 
$$R(X;D_1,\dots,D_r,A)$$ 
is finitely generated, and hence, the cone $\mcal N=\pi^{-1}\big(\Nef(X)\big)\cap\sum\R_+D_i$ is rational polyhedral by Proposition \ref{cor:5}, where $\pi\colon\Div_\R(X)\to N^1(X)_\R$ is the natural map. Since $\pi(D)\in\mcal N$, the result follows.

Now assume the Cone conjecture in dimension $n$. Let $D$ be an $\R$-divisor whose class is in $\Nef(X)\cap\Eff(X)$, and let $\Pi$ be the fundamental domain for the action of $\mcal A(X)$ on $\Nef(X)\cap\Eff(X)$. Then there exists $g\in\mcal A(X)$ such that $D\in g\Pi$, and the conclusion follows since $g\Pi$ is a rational polyhedral cone.
\end{proof}

\bibliographystyle{amsalpha}

\bibliography{biblio}
\end{document}